\documentclass{article} 
\usepackage{amssymb,amsmath,amsthm}
\usepackage{amscd}
\usepackage{hyperref}
\usepackage{epsfig}
\usepackage{amsfonts}
\usepackage{amssymb}
\usepackage{amsmath,enumerate}
\usepackage{commath}
\usepackage{euscript}
\usepackage{enumitem}
\usepackage{amscd}
\usepackage{xcolor}
\usepackage[ruled,vlined]{algorithm2e}
\usepackage[numbers,sort&compress]{natbib}

\newtheorem{theorem}{Theorem}[section]

\newtheorem{definition}{Definition}[section]

\newtheorem{lemma}{Lemma}[section]

\newtheorem{example}{Example}[section]

\makeatletter
\newsavebox{\@brx}
\newcommand{\llangle}[1][]{\savebox{\@brx}{\(\m@th{#1\langle}\)}%
	\mathopen{\copy\@brx\kern-0.5\wd\@brx\usebox{\@brx}}}
\newcommand{\rrangle}[1][]{\savebox{\@brx}{\(\m@th{#1\rangle}\)}%
	\mathclose{\copy\@brx\kern-0.5\wd\@brx\usebox{\@brx}}}
\makeatother
\begin{document}
	\title{Fixed Point for Uniformly Local Asymptotic Nonexpansive Map}
	\author{
		Pallab Maiti\footnotemark[2], Asrifa Sultana\footnotemark[1] \footnotemark[2]}
	\date{ }
	\maketitle
	\def\thefootnote{\fnsymbol{footnote}}
	
	\footnotetext[1]{ Corresponding author. e-mail- {\tt asrifa@iitbhilai.ac.in}}
	\noindent
	\footnotetext[2]{Department of Mathematics, Indian Institute of Technology Bhilai, Raipur - 492015, India.
	}
	
	
		
	\begin{abstract}
	Fixed points for uniformly local asymptotic nonexpansive maps are discussed in this article. An approximate fixed point sequence for such a map over a uniformly convex Banach space is derived. At the end, we study the unique fixed point for uniformly local asymptotic contraction.
	\end{abstract}
	{\bf Keywords:}
	Asymptotically nonexpansive map, fixed points, Uniformly convex Banach space, Uniformly local map.\\
	{\bf Mathematics Subject Classification:}
	47H10
	
\section{Introduction}\label{sec1}

In the year 1965, Browder \cite{brow} extended the famous Banach contraction principle for nonexpansive maps \cite{brow}. A map that satisfies the Lipschitz condition is called nonexpansive whenever the Lipschitz constant is equal to one. In the year 1972, Goebel and Kirk \cite{goebel} improved the Browder's \cite{brow} result through the perception of asymptotically nonexpansive mappings \cite{goebel}. A self map $T$ over a non-void subset $Q$ of a normed linear space $U$ is asymptotically nonexpansive, if for all $p,q\in Q$ and any $n$ lies in $\mathbb{N}$,
\begin{equation}
	\|T^np-T^nq\|\leq \beta_n\|p-q\|,
\end{equation}
where $\{\beta_n\}_n$ is a sequence of positive real numbers having $\lim_{n\to \infty}\beta_n=1$. The authors \cite{goebel} established that if $T$ is asymptotically nonexpansive map over a non-void closed convex bounded subset of a uniformly convex Banach space, then the map $T$ possess a fixed point. There are several generalization of this result, one can found in \cite{khamsi,bose,kozlowski,xu,schu}.

On other side, Edelstein \cite{edel1} generalized the elegant Banach contraction principle for uniformly local contraction \cite{edel1} on a metric space, in the year 1969. A self map $T:(U,\|.\|)\to (U,\|.\|)$ is described as a uniformly local contraction \cite{edel} if for each $p,q\in U$ having $\|p-q\|<r,\,(r>0)$ fulfils
\begin{equation}\label{edel_norm}
	\|Tp-Tq\|\leq \theta \|p-q\|~\textrm{where}~\theta\in [0,1).
\end{equation}
It is worth to note that any mapping which satisfies the classical contraction due to Banach follows from uniformly local contraction due to Edelstein \cite{edel}, whereas the converse statement may not be true. For more results about the presence of fixed points for uniformly local contraction, one may view in \cite{Jach,mam}.

Inspired by these results, in this article, we derive sufficient criteria for the occurrence of fixed points for a uniformly local asymptotic nonexpansive map, which is indeed a generalization of Goebel-Kirk's \cite{goebel} result on a normed linear space. Our main theorem also generalized the Kirk-Xu's \cite[Theorem 3.5]{xu} result about the occurrence of fixed points for pointwise asymptotic nonexpansive map \cite{xu}. We also discussed a new type of the Demiclosedness Principle \cite{kozlowski} for uniformly local asymptotic nonexpansive mappings. The approximate fixed point sequence for uniformly local asymptotic nonexpansive maps are studied. Further, the unique fixed points for uniformly asymptotic local contractions are derived in this article.
\section{Preliminaries}\label{sec2}
This segment contains some mathematical notations and definitions, which are useful up to the end of this paper.

Consider a non-void subset $Q$ of the set $U$. For any bounded sequence $\{q_n\}_n\in U$, we define a map $f:Q\to \mathbb{R}$ by $f(q)=\limsup_{n\to \infty}\|q_n-q\|$ for each $q\in Q$. Then the asymptotic centre \cite{xu} of the sequence $\{q_n\}_n$ relative to $Q$ is defined by $$C_0=\{q\in Q: f(q)=\inf_{z\in Q}f(z)\},$$ and asymptotic radius by $$\rho=\inf\{\limsup_{n\to \infty}\|q_n-y\|:y\in Q\}.$$

The below-stated lemma due to Alfuraidan and Khamsi \cite{khamsi} ensures that the collection $C_0$ is non-void.
\begin{lemma}\cite{khamsi}
	Let $(U,\|.\|)$ be a normed linear space and $Q\subseteq U$ be a non-void convex closed. Suppose that $\{q_n\}_n\in U$ is bounded and $f:Q\to [0,\infty)$ is a mapping defined by 
	\begin{equation}
		f(q)=\limsup_{n\to \infty}\|q_n-q\|~\text{ for each}~ q\in Q.
	\end{equation}
	Then $C_0$ is non-void if $(U,\|.\|)$ is complete and uniformly convex.
\end{lemma}
The upcoming lemma consists of a property of uniformly convex Banach space, which is very useful for establishing our main result.
\begin{lemma}\cite{xu}\label{lem1}
	A complete normed linear space $(U,\|.\|)$ becomes a uniformly convex if and only if for every fixed scalar $a>0$, there is a continuous function $\phi:[0,\infty)\to [0,\infty)$ having $\phi^{-1}(0)=\{0\}$ and 
	$$\|cp+(1-c)q\|^2\leq c\|p\|^2+(1-c)\|q\|^2-c(1-c)\phi(\|p-q\|),$$
	where $0\leq c\leq 1$ and $p,q\in U$ fulfils $\|p\|\leq a$, $\|q\|\leq a$.
\end{lemma}
Let us review the concept of Opial property \cite{opial}, which is needed in the sequel. A complete normed linear space $U$ fulfils the Opial property \cite{opial} if for every sequence $\{q_n\}_n\in U$ converges weakly to $q\in U$, and for each $p\in U$ having $p\neq q$ implies 
\begin{equation*}
	\limsup_{n\to \infty}\|q_n-q\|\leq \limsup_{n\to \infty}\|q_n-p\|.
\end{equation*}
All the Hilbert spaces and $l^p(1\leq p<\infty)$ satisfies this property, where $L^p$ does not fulfil Opial property except $p=2$.  
The below mentioned lemma due to G\'ornicki \cite{gorn} is necessary to establish our result in the sequel.
\begin{lemma}\cite{gorn}\label{lem2}
	Suppose that $Q\subseteq U$ is closed convex and $U$ is uniformly Banach space having Opial's condition. If $\{q_n\}_n\in Q$ converges to $w$ weakly, then $w$ lies in asymptotic centre of $\{q_n\}_n$.
\end{lemma}	
\section{Main results}\label{sec3}
We commence this section with the concept of a uniformly local asymptotic nonexpansive map, which is in fact an extension of the asymptotically nonexpansive map \cite{goebel} and we derive fixed points for such mappings on a normed linear space. 
\begin{definition}
	For a normed linear space $(Q,\|.\|)$, a map
	$T:Q\to Q$ is defined as uniformly local asymptotic nonexpansive if for $r>0$ and  for each $p,q\in Q$ having $\|p-q\|<r$ implies,
	\begin{equation}\label{local_nonexpansive}
		\|T^np-T^nq\|\leq \beta_n\|p-q\|~\text{for any}~n\in \mathbb{N},
	\end{equation}
	where $\{\beta_n\}_n$ is a sequence of positive real number converges to $1$.
\end{definition}
In the below, we provide an example of uniformly local asymptotic nonexpansive map and it further shows that uniformly local asymptotic nonexpansive map is indeed an extension of asymptotic nonexpansive map due to Goebel and Kirk \cite{goebel}.
\begin{example}
	Let $V=(l^2,\|.\|_2)$ be a normed linear space and $U=\{u\in l^2:\|u\|_2\leq \frac{1}{2} \}\cup \{(1,0,0,\cdots)\}\subseteq V$. We define a map $T:U\to U$ by 
	\begin{eqnarray*}
		T\left(u_1,u_2,\cdots\right)&=&\left(0,0,u_1^2,b_2u_2,b_3u_3,\cdots\right),\\
		T\left(1,0,0,\cdots\right)&=&\left(\frac{1}{2},0,0,0,\cdots\right),
	\end{eqnarray*}
	where $\{b_n\}_{n\geq 2}$ is a decreasing sequence and for every $n\geq 2$, $b_n\in (0,1)$ with $\prod_{j=1}^{\infty}b_{2j+1}=1$. Now, for each $n\in \mathbb{N}$, $T^n(0,0,\cdots)=(0,0,\cdots)$ and $T^n\left(1,0,0,\cdots\right)=(0,0\cdots,0,\frac{b_{2n-3}b_{2n-5}\cdots b_5b_3}{2^2},0,\cdots)$. Consequently we have, $\|T^n\left(1,0,0,\cdots\right)-T^n(0,0,\cdots)\|_2=\beta_n\|\left(1,0,0,\cdots\right)-(0,0,\cdots)\|_2$, where $\beta_n=\frac{1}{2^2}\prod_{j=1}^{n}b_{2j+1}$ and $n\to \infty$, $\beta_n\not\to 1$. 
	
	Again, for each $u\in U$ and for every positive integer $n$, it occurs that $T^n(u_1,u_2,\cdots)=(0,\cdots,0,b_{2n-1}\cdots b_3u_1^2,b_{2n}\cdots b_2u_2,\cdots )$. Now for any $u,v\in U$ with $\|u-v\|<\frac{1}{2}$ it yields that,
	\begin{equation*}
		\|T^nu-T^nv\|\leq \beta_n\|u-v\|,
	\end{equation*}
	where $\beta_n=\prod_{j=1}^{n}b_{2j+1}$. Hence the map $T$ is a uniformly local asymptotic nonexpansive.
\end{example}

The upcoming theorem ensures the existence of fixed points for the mappings that meets the criteria (\ref{local_nonexpansive}).
\begin{theorem}\label{thm1}
	Let us assume $(U,\|.\|)$ is a uniformly convex Banach space and $Q\subseteq U$ is non-void bounded closed convex. Suppose that $T:Q\to Q$ is a uniformly local asymptotic nonexpansive. Then the occurrence of fixed point for $T$ is guaranteed if there is $q_0\in Q$ in order that the asymptotic radius $\rho$ of $\{T^nq_0\}_n$ is less than $r$.
	
	Furthermore, the collection of fixed points of $T$ is a closed set. 
\end{theorem}
\begin{proof}
	For fixed $q_0$, consider a map $f:Q\to [0,\infty)$ in order that $f(q)=\limsup_{n\to \infty}\|T^nq_0-q\|$ for $q\in Q$. Consequently, we obtain  $w\in Q$ in order that $f(w)=\min_{q\in Q}f(q)$ on the account of $Q$ is weakly compact subset of uniformly convex Banach space $U$. Due to the fact that asymptotic radius corresponding to $\{T^nq_0\}_n$ is less than $r$, then $\limsup_{n\to \infty}\|T^nq_0-w\|<r$. As a consequence we find $N_0$ in $\mathbb{N}$ meets $\|w-T^nq_0\|<r$, for every $n\geq N_0$. Indeed, if it is not hold, then there is $\{T^{n_k}q_0\}_k$ such that $\|T^{n_k}q_0-w\|\geq r$ for each $k\in \mathbb{N}$. Again $\limsup_n\|w-T^nq_0\|\geq \limsup_k \|T^{n_k}q_0-w\|\geq r$, which leads to a contradiction. 
	
	As $q_0,\,w\in Q$ and $Q$ is bounded, then by applying the Lemma \ref{lem1} for $l,m\in \mathbb{N}$  we obtain,
	\begin{multline*}
		\|T^{l+m+n}q_0-\frac{1}{2}(T^lw+T^mw)\|^2\leq\frac{1}{2}\|T^{l+m+n}q_0-T^lw\|^2\\
		+ \frac{1}{2}\|T^{l+m+n}q_0-T^mw\|^2-\frac{1}{4}\phi(\|T^lw-T^mw\|).
	\end{multline*}
	On the account of $\|w-T^nq_0\|<r$, for all $n\geq N_0$, we observe that $\|w-T^{m+n}q_0\|<r$ and $\|w-T^{l+n}q_0\|<r$ for each $n\geq N_0$ and $l,m\in \mathbb{N}$.
	As $T$ is locally asymptotic nonexpansive then form the last inequality yields that,
	\begin{multline*}
		\|T^{l+m+n}q_0-\frac{1}{2}(T^lw+T^mw)\|^2\leq\frac{1}{2}(\beta_l)^2\|w-T^{m+n}q_0\|^2\\
		+\frac{1}{2}(\beta_m)^2\|w-T^{l+n}q_0\|^2-\frac{1}{4}\phi(\|T^lw-T^mw\|).
	\end{multline*}
	Now taking limit supremum as $n\to \infty$ we have
	\begin{equation*}
		f\left(\frac{T^mw+T^lw}{2}\right)\leq \frac{1}{2}\left\{(\beta_l)^2+(\beta_m)^2\right\}f(w)-\frac{1}{4}\phi(\|T^lw-T^mw\|).
	\end{equation*}
	Due to the fact that $w$ is the minimizer of $f$, then 
	$$f(w)\leq \frac{1}{2}\left\{(\beta_l)^2+(\beta_m)^2\right\}f(w)-\frac{1}{4}\phi(\|T^lw-T^mw\|).$$
	Now taking $l,m\to \infty$ we get that $\{T^nw\}_n$ is Cauchy. As $Q$ is a closed subset of a complete normed linear space $U$, then $\{T^nw\}_n$ is convergent. Let $\{T^nw\}_n$ converges to $q^*\in Q$.

	Because of $T^nw\to q^*$, then we achieve $N\in \mathbb{N}$ in order that $\|T^nw-q^*\|<r$ for each $n\geq N$. Now for each $n\geq N$, 
	$$\|T^{n+1}w-Tq^*\|\leq \beta_1\|T^nw-q^*\|.$$
	Taking $n\to \infty$, the last inequation yields that $Tq^*=q^*$.
	
	Next our claim that the collection of fixed points $Fix(T)$ is a closed set. Choose $\{p_n\}_n$ in $Fix(T)$ and $p_n\to p$. As a consequence we obtain $M_0\in \mathbb{N}$ in order that $\|p_n-p\|<r\,\,\forall n\geq M_0$. Now for each $n\geq M_0$,
	\begin{eqnarray*}
		\|Tp-p\|&\leq& \|p-p_n\|+\|p_n-Tp\|\\
		&\leq&\|p-p_n\|+\|Tp_n-Tp\|\qquad[\textrm{$\because ~p_n\in Fix(T)$}]\\
		&\leq& \|p-p_n\|+\beta_1\|p_n-p\|.
	\end{eqnarray*}
	Taking limit $n\to\infty$, we conclude that $p\in Fix(T)$. Thus the collection of fixed point is a closed set.
\end{proof}
Kirk and Xu \cite[Theorem 3.5]{xu} generalized the Goebel-Kirk \cite{goebel} result for the pointwise asymptotically nonexpansive map \cite{xu} in the year 2008. In the succeeding theorem, we improved the Kirk-Xu's \cite{xu} result for uniformly local pointwise asymptotic nonexapnsive map and the proof line is followed from the aforementioned Theorem \ref{thm1}.
\begin{theorem}
	Let us assume $(U,\|.\|)$ is a uniformly convex Banach space and $Q\subseteq U$ is non-void bounded closed convex. Suppose that $T:Q\to Q$ in order that for each $p,q\in Q$ having $\|p-q\|<r$ implies
	\begin{equation*}
		\|T^np-T^nq\|\leq \alpha_n(q)\|p-q\|~\text{for any}~n\in \mathbb{N},
	\end{equation*}
	where $\alpha_n\to 1$ pointwise on $Q$. Then the occurrence of fixed point for $T$ is guaranteed if there is $q_0\in Q$ in order that the asymptotic radius $\rho$ of $\{T^nq_0\}_n$ is less than $r$.
\end{theorem}
In the year 2011, Kozlowski \cite{kozlowski} derived the existence of fixed points for the mappings satisfied asymptotically pointwise nonexpansive \cite{kozlowski} criteria on a uniformly convex Banach space having Opial \cite{kozlowski} property. The author \cite{kozlowski} deduced the below mentioned theorem.
\begin{theorem}\cite{kozlowski}\label{koz}
	Suppose that $Q$ is a non-void bounded closed convex subset of a uniformly convex Banach space $U$ having the Opial property and a map $T:Q\to Q$ in order that for every $p,q\in Q$,
	\begin{equation}\label{koz_eq}
		\|T^np-T^nq\|\leq \alpha_n(q)\|p-q\|~\text{for any}~n\in \mathbb{N},
	\end{equation}
	where $\limsup_{n\to \infty}\alpha_n(q)\leq 1$ and $\sum_{n=1}^{\infty}(\alpha_n(q)-1)<\infty$ for each $q\in Q$. Also assume that, $\{q_n\}_n\in Q$ in order that $q_n\xrightarrow[]{w}w$ for some $w\in Q$ and $\|Tq_n-q_n\|\to 0$. Then $w$ is the fixed point of $T$.	
\end{theorem}
Now we improve the above mentioned Theorem \ref{koz} due to Kozlowski \cite{kozlowski} for the mappings satisfied the equation $(\ref{koz_eq})$ locally. For this purpose we establish the upcoming lemma.
\begin{lemma}\label{helping lemma}
	Let $Q$ be a non-void subset of a normed linear space $(U,\|.\|)$ and a map $T:Q\to Q$ fulfils the equation $(\ref{local_nonexpansive})$. Let $\{q_n\}_n\in Q$ such that $\|Tq_n-q_n\|\to 0$ as $n\to \infty$. Then for each $m\in \mathbb{N}$, $\|T^mq_n-q_n\|\to 0$ as $n\to \infty$.
\end{lemma}

\begin{proof}
	On the account of $\|Tq_n-q_n\|\to 0$ as $n\to \infty$, then there is $N_0\in \mathbb{N}$ in order that $\|Tq_n-q_n\|<r$ for each $n\geq N_0$. Now for every $n\geq N_0$,
	\begin{small}
		\begin{equation*}
			\begin{split}
				\|T^mq_n-q_n\|&\leq\|T^mq_n-T^{m-1}q_n\|+\|T^{m-1}q_n-T^{m-2}q_n\|+\cdots+\|Tq_n-q_n\|\\
				&\leq \sum_{j=2}^{m-1}\|T^{j-1}q_n-T^jq_n\|+\|Tq_n-q_n\|\\
				&\leq \left(1+\sum_{j=2}^{m-1}\beta_j\right)\|Tq_n-q_n\|.~~[\because\|Tq_n-q_n\|<r]
			\end{split}
		\end{equation*}
	\end{small}	
	As the sequence $\{\beta_n\}_n$ converges to $1$, then it is bounded, hence there is $M>0$ such that $\sum_{j=1}^{m-1}\beta_j\leq M$. Consequently, for every $n\geq N_0$, the last inequality leads to
	\begin{equation*}
		\|T^mq_n-q_n\|\leq (1+M)\|Tq_n-q_n\|.
	\end{equation*}
	Now taking $n\to \infty$, we conclude that $\|T^mq_n-q_n\|\to 0$.
\end{proof}
With the help of the Lemma \ref{lem2}, in the succeeding theorem, we scrutinize fixed points for the mappings satisfying the equation $(\ref{koz_eq})$ locally. This result is in fact an extension of the Theorem \ref{koz} due to Kozlowski \cite{kozlowski}.

\begin{theorem}
	Let $Q$ be a closed convex bounded subset of a uniformly convex Banach space $U$ having Opial's condition. Suppose that a map $T:Q\to Q$ fulfils the equation $(\ref{local_nonexpansive})$.  Let $\{q_n\}_n\in Q$ having $q_n\xrightarrow[]{w}w$ and $q_n-Tq_n\to 0$, then $w$ lies is in the set of fixed point of $T$ if the asymptotic radius of $\{q_n\}_n$ relative $Q$ is less than $r$.
\end{theorem}
\begin{proof}
	Let us define a type function $f:Q\to [0,\infty)$ by $f(q)=\limsup_{n}\|q_n-q\|$ for each $q\in Q$. As $q_n\xrightarrow[]{w}w$, then by the Lemma \ref{lem2}, $w$ resides in the asymptotic centre of $\{q_n\}_n$ relative to $Q$. That is,
	\begin{center}
		$f(w)=\inf_{q\in Q}f(q)$.
	\end{center}
	Again by our assumption, the asymptotic radius of $\{q_n\}_n$ relative to $Q$ is less than $r$, hence $\inf_{q\in Q}f(q)<r$, which indicates that $f(w)<r$. That is $\limsup_{n}\|q_n-w\|<r$. As a consequence we obtain $N_0\in \mathbb{N}$ such that $\|q_n-w\|<r$ for each $n\geq N_0$.
	
	On the other hand, $q_n-Tq_n\to 0$ as $n\to \infty$, therefore there is $N_1\in \mathbb{N}$ in order that $\|q_n-Tq_n\|<r$ for every $n\geq N_1$. Now for fix $m\in \mathbb{N}$ and $m>2$, then for $n\geq N_1$, 
	\begin{eqnarray*}
		\|T^mq_n-q\|&\leq &\|T^mq_n-T^{m-1}q_n\|+\cdots+\|q_n-Tq_n\|+\|q_n-q\|\\
		&\leq& \beta_{m-1}\|q_n-Tq_n\|+\cdots+\|Tq_n-q_n\|+\|q_n-q\|\\
		&=&\left(1+\sum_{i=1}^{m-1}\beta_{i}\right)\|Tq_n-q_n\|+\|q_n-q\|.
	\end{eqnarray*}
	Consequently,
	\begin{equation}\label{eqn2}
		\limsup_{n}\|T^mq_n-q\|\leq \limsup_{n}\|q_n-q\|=f(q).
	\end{equation}
	Again using the Lemma \ref{helping lemma}, we have 
	\begin{equation}\label{eqn3}
		\begin{aligned}\nonumber
			f(q)=\limsup_{n}\|q_n-q\|&\leq \limsup_{n}\|q_n-T^mq_n\|+\limsup_{n}\|T^mq_n-q\|\\
			&\leq \limsup_{n}\|T^mq_n-q\|.
		\end{aligned}
	\end{equation}
	Hence $(\ref{eqn2})$ and last inequation leads to $f(q)=\limsup_{n}\|T^mq_n-q\|$ for each $q\in Q$.
	Now for every $n\geq N_0$ and $m>2$,
	\begin{eqnarray*}
		f(T^mw)&=&\limsup_{n}\|T^mq_n-T^mw\|\\
		&\leq & \beta_m\limsup_{n}\|q_n-w\|\qquad\text{[$\because\|q_n-w\|<r$]}\\
		&=&\beta_m f(w).
	\end{eqnarray*}
	Taking limit $m\to\infty$ in the last inequality, 
	\begin{equation}\label{eqn4}
		\lim_{m\to \infty}f(T^mw)\leq f(w).
	\end{equation}
	Since $w$ is the minimizing point of $f$ in $Q$ then we conclude that 
	\begin{equation}\label{eqn5}
		\lim_{m\to \infty}f(T^mw)=f(w).
	\end{equation}
	As $U$ is uniformly convex Banach space, then applying the Lemma \ref{lem1}, we have a map $\phi:[0,\infty)\to [0,\infty)$ fulfils the following
	\begin{small}
		\begin{equation*}
			\|q_n-\frac{1}{2}\left(w+T^mw\right)\|^2\leq \frac{1}{2}\|q_n-w\|^2+ \frac{1}{2}\|q_n-T^mw\|^2-\frac{1}{4}\phi\left(\frac{1}{2}\|T^mw-w\|\right).
		\end{equation*} 
	\end{small}
	Taking limit sup as $n\to \infty$ it yields that, 
	\begin{equation*}
		f(w)^2\leq \frac{1}{2}f(w)^2+ \frac{1}{2}	f(T^mw)^2-\frac{1}{4}\phi\left(\frac{1}{2}\|T^mw-w\|\right).
	\end{equation*}
	Now take $m\to \infty$, we have $\phi\left(\frac{1}{2}\|T^mw-w\|\right)=0$, hence using the continuity of $T$ we conclude that $Tw=w$.
\end{proof}
The well known Picard's method creates a sequence $\{q_n\}_n\in U$, successively by $q_{n+1}=Tq_n$ for every $n\in \mathbb{N}$ and Banach established that this sequence converges to a unique fixed point of $T$ if the map $T$ is contraction on complete metric space $U$. In the year 1991, a more generalized iterative scheme $q_{n+1}=(1-\alpha_n)q_n+\alpha_nT^nq_n$, where $\alpha_n\in (0,1)$, $n\in \mathbb{N}$, and $T$ is asymptotically nonexpansive map, was introduced by Schu \cite{schu}.  For this sequence the author \cite[Theorem 1.4]{schu} derived the occurrence of approximate fixed point of the map $T$ on a subset of a Hilbert space. In the below stated theorem, we extend this iterative sequence by considering the map $T$ as uniformly local asymptotically nonexpansive and we further establish the presence of approximate fixed point for the same map $T$.
\begin{theorem}
	Let $Q$ be a non-void bounded closed convex subset of a uniformly convex Banach space $(U,\|.\|)$. Suppose that $T:Q\to Q$ is a uniformly local asymptotic nonexpansive map. Let $q_1\in Q$ and $q_{n+1}=(1-\gamma_n)q_n+\gamma_nT^nq_n$, where $0<a\leq \gamma_n\leq b<1$ and $\sum_{n\in \mathbb{N}}\gamma_n<\infty$. Then $\lim_{n\to \infty}\|q_n-Tq_n\|=0$, if there is $q_0\in Q$ such that the asymptotic radius $\rho$ of $\{T^nq_0\}_n$ is less than $r$.
\end{theorem}
\begin{proof}
	By the Theorem \ref{thm1}, we get that $q^*\in Q$ in order that $Tq^*=q^*$. Now, 
	\begin{eqnarray}
		\|q_{n+1}-q^*\|^2
		&=&\|(1-\gamma_n)(q_n-q^*)+\gamma_n(T^nq_n-q^*)\|^2.\label{eqn1}
	\end{eqnarray}
	Then by applying the Lemma \ref{lem1} on the equation $(\ref{eqn1})$, we obtain
	\begin{multline*}
		\|q_{n+1}-q^*\|^2\leq (1-\gamma_n)\|q_n-q^*\|^2+\gamma_n\|T^nq_n-q^*\|^2\\-\gamma_n(1-\gamma_n)\phi(\|T^nq_n-q_n\|),
	\end{multline*}
	which implies that 
	\begin{multline*}
		\gamma_n(1-\gamma_n)\phi(\|T^nq_n-q_n\|)\leq 
		\|q_n-q^*\|^2-\gamma_n\|q_n-q^*\|^2+\gamma_n\|T^nq_n-q^*\|^2\\
		-\|q_{n+1}-q^*\|^2.
	\end{multline*}
	Then using the bound of $\{\gamma_n\}_n$ we have
	\begin{equation*}
		a(1-b)\phi(\|T^nq_n-q_n\|)\leq \|q_n-q^*\|^2+\gamma_n\|T^nq_n-q^*\|^2-\|q_{n+1}-q^*\|^2.
	\end{equation*}
	Assume that $M=\{diam(Q)\}^2$. Consequently for each $n$,
	\begin{equation*}
		a(1-b)\phi(\|T^nq_n-q_n\|)\leq \|q_n-q^*\|^2+\gamma_nM-\|q_{n+1}-q^*\|^2.
	\end{equation*} 
	Now,
	\begin{small}
		\begin{equation*}
			\begin{split}
				a(1-b)\sum_{1}^{m}\phi(\|T^nq_n-q_n\|)&\leq\sum_{1}^{m}\left[\|q_n-q^*\|^2-\|q_{n+1}-q^*\|^2\right]
				+M\sum_{1}^{m}\gamma_n\\
				&=\|q_1-q^*\|^2-\|q_{m+1}-q^*\|^2 +M\sum_{1}^{m}\gamma_n\\
				&\leq\|q_1-q^*\|^2+M\sum_{1}^{m}\gamma_n.
			\end{split}
		\end{equation*}
	\end{small}
	Now as $m\to \infty$, we have $\sum_{1}^{\infty}\phi(\|T^nq_n-q_n\|)<\infty$. Hence the limiting value of $\phi(\|T^nq_n-q_n\|)\to 0$ implies that $\lim_{n\to \infty}\|T^nq_n-q_n\|=0$.
	
	Again, $\|q_{n+1}-q_n\|=\| (1-\gamma_n)q_n+\gamma_nT^nq_n-q_n\|=\gamma_n\|q_n-T^nq_n\|$. Then it follows that $\lim_{n\to \infty}\|q_{n+1}-q_n\|=0$. As a result there is $N_1\in \mathbb{N}$ such that $\|q_{n+1}-q_n\|<r$ for all $n\geq N_1$. Since $\lim_{n\to \infty}\|T^nq_n-q_n\|=0$, then there is $N_2\in \mathbb{N}$ in order that $\|T^nq_n-q_n\|<r$ for each $n\geq N_2$. Consider $N_3=\max\{N_1,N_2\}$, consequently 
	$$\|q_{n+1}-q_n\|<r,~~\|T^nq_n-q_n\|<r~~\forall n\geq N_3.$$
	Now for every $n\geq N_3$ we have
	\begin{eqnarray*}
		\|q_n-Tq_n\|&\leq& \|q_{n}-q_{n+1}\|+\|q_{n+1}-T^{n+1}q_{n+1}\|\\
		&\quad&+\|T^{n+1}q_{n+1}-T^{n+1}q_n\|+\|T^{n+1}q_n-Tq_n\|\\
		&\leq& \|q_{n+1}-q_n\|+\|q_{n+1}-T^{n+1}q_{n+1}\|\\
		&\quad&+\beta_{n+1}\|q_{n+1}-q_n\|+\beta_1\|T^nq_n-q_n\|
	\end{eqnarray*}
	Now $n\to \infty$ in the last inequality we have $\lim_{n\to \infty}\|q_n-Tq_n\|=0$.
\end{proof}
In the below stated part, we also derive the unique fixed point for uniformly local asyptotic contraction. 	

\begin{definition}
	A mapping $T:Q\to Q$ is said to be an uniformly asymptotic local contraction if for each $p,q\in Q$ and $n\geq 1$,
	\begin{equation}\label{local asyptotic contraction}
		\|p-q\|<r\implies\|T^np-T^nq\|\leq \beta_n\|p-q\|,
	\end{equation}
	where $\beta_n\to \beta$ and $\beta\in [0,1)$.
\end{definition}
Now we find a sufficient criteria for the occurrence of unique fixed point of any mappings meet the condition (\ref{local asyptotic contraction}). 

\begin{theorem}\label{thm1.1}
	Let $Q$ be a weakly compact convex subset of a Banach space $U$ and a map $T:Q\to Q$ be a uniformly asymptotic local contraction. Then $T$ possesses a unique fixed point if there is $q_0\in Q$ in order that the asymptotic radius of $\{T^nq_0\}_n$ relative to $Q$ is less than $r$.
\end{theorem}
\begin{proof}
	Define a type function $f:Q\to [0,\infty)$ in order that 
	\begin{equation}\label{eqn1.1}
		f(q)=\limsup_{n}\|T^nq_0-q\|~\textrm{for each}~ q\in Q.
	\end{equation}
	Due to the fact that $Q$ is weakly compact convex in the Banach space $U$, then the set  $A_C\{T^nq_0\}=\{w\in Q:f(w)=\min_{z\in Q}f(z)\}$ is non-void.
	
	Let $w\in A_C\{T^nq_0\}$, then $f(w)=\min_{x\in Q} f(x)$. Since the asymptotic radius of $\{T^nq_0\}_n$ is less than $r$, then we have $f(w)<r$, that is $\limsup_{n}\|T^nq_0-w\|<r$. As a consequence we obtain $N_0\in \mathbb{N}$ such that $\|T^nq_0-w\|<r$ for all $n\geq N_0$. Now for each $m\in \mathbb{N}$ and $n\geq N_0$,
	\begin{eqnarray*}
		f(T^mw)&=&\limsup_{n}\|T^nq_0-T^mw\|\\
		&=& \limsup_{n}\|T^{n+m}q_0-T^mw\|\\
		&\leq & \beta_m\limsup_{n}\|T^nq_0-w\|\qquad[\because\|T^nq_0-w\|<r]\\
		&=&\beta_mf(w).
	\end{eqnarray*}
	Therefore $f(w)\leq f(T^mw)\leq \beta_mf(w)$. Then taking limit $m\to \infty$, we have $f(w)\leq \beta f(w)<f(w)$, which is a contradiction. Hence $f(w)=0$ and consequently $f(T^mw)=0$ for each $m\in \mathbb{N}$. As  a result we get that $T^nq_0\to w$ and $T^nq_0\to Tw$, whenever $n\to \infty$. This implies that $Tw=w$.

	As $\beta<1$, then there is a positive $\varepsilon$ such that $\beta+\varepsilon<1$. Again $\beta_n\to \beta$, then there is $N_0\in \mathbb{N}$ meets $\beta_n<\beta+\varepsilon<1$ for each $n\geq N_0$.  
	
	Let $u,v\in Q$ be two fixed points of the map $T$. On the account of $U$ is a Banach space, then there is a finite path $(z_i)_{i=0}^L$ between $u$ and $v$ in order that $z_0=u$, $z_N=v$ and $\|z_{i-1}-z_i\|<r$ for every $i=1,2,\cdots, L$.  Now for each $n\geq N_0$ and $i$, we see that 
	\begin{equation*}
		\|T^nz_{i-1}-T^n{z_i}\|\leq \beta_n\|z_{i-1}-z_i\|<(\beta+\varepsilon)\|z_{i-1}-z_i\|<r.
	\end{equation*}
	Now for every $n\geq N_0$ and $m\in \mathbb{N}$,
	
	\begin{eqnarray*}
		\begin{aligned}
			\|u-v\|&=\|T^{nm}u-T^{nm}v\|\nonumber\\&= \|T^{nm}z_0-T^{nm}z_L\|\nonumber\\
			&\leq \|T^{nm}z_0-T^{nm}z_1\|+\cdots+\|T^{nm}z_{L-1}-T^{nm}z_L\|\\
			&\leq (\beta+\varepsilon)\|T^{(m-1)n}z_0-T^{(m-1)n}z_1\|+\cdots\cdots\\
			&\qquad\cdots\cdots+(\beta+\varepsilon)\|T^{(m-1)n}z_{L-1}-T^{(m-1)n}z_L\|\\
			&\cdots\hspace*{0.5cm}\cdots\hspace*{0.5cm}\cdots\\
			&\leq (\beta+\varepsilon)^m\|z_0-z_1\|+\cdots+(\beta+\varepsilon)^m\|z_{L-1}-z_L\|.
		\end{aligned}
	\end{eqnarray*}
	Taking $m\to \infty$ in the last inequation, we have $\|u-v\|=0$, that is $u=v$.
\end{proof}
Krik and Xu \cite[Theorem 3.1]{xu} also deduced the existence of fixed point for pointwise asymptotic contraction maps. In the succeeding theorem, we discuss about the occurrence of fixed point for uniformly local pointwise asymptotic contractive map and the proof line is followed from the aforementioned Theorem \ref{thm1.1}.
\begin{theorem}
	Let $Q$ be a weakly compact convex subset of a Banach space $U$ and $T:Q\to Q$ be a map in order that for each $p,q\in Q$ and $n\geq 1$,
	\begin{equation*}
		\|p-q\|<r\implies\|T^np-T^nq\|\leq \beta_n(p)\|p-q\|,
	\end{equation*}
	where $\beta_n\to \beta$ pointwise on $Q$ and $\beta:Q\to [0,1)$. Then $T$ possesses a unique fixed point if there is $q_0\in Q$ in order that the asymptotic radius of $\{T^nq_0\}_n$ relative to $Q$ is less than $r$.
\end{theorem}
%

\section*{Acknowledgment}

The first author would like to acknowledge the Ministry of Human Resource Development, India for providing financial assistance during the research work.

\end{document}